\documentclass{amsart}

\usepackage{amsmath,amssymb,amsthm}
\usepackage{graphicx}
\usepackage{fullpage}
\usepackage{hyperref}
\usepackage{mathtools}
\usepackage{tikz}
\usetikzlibrary{matrix,arrows}
\usepackage{color}

\newtheoremstyle{myremark} 
    {7pt}                    
    {7pt}                    
    {}  	                 
    {}                           
    {\bf}       	         
    {.}                          
    {.5em}                       
    {}  

\theoremstyle{plain}
\newtheorem{lemma}{Lemma}[section]
\newtheorem{theorem}[lemma]{Theorem}
\newtheorem*{theorem-main}{Main Theorem}
\newtheorem{definition}[lemma]{Definition}
\newtheorem{corollary}[lemma]{Corollary}
\newtheorem{proposition}[lemma]{Proposition}
\newtheorem{conjecture}[lemma]{Conjecture}

\theoremstyle{definition}

\theoremstyle{myremark}
\newtheorem{remark}[lemma]{Remark}
\newtheorem{example}[lemma]{Example}

\newcommand{\N}{\mathbb{N}}

\newcommand{\R}{\mathbb{R}}

\newcommand{\Z}{\mathbb{Z}}

\newcommand{\cP}{\mathcal{P}}

\newcommand{\diam}{\mathrm{diam}}

\newcommand{\conv}{\mathrm{conv}}
\newcommand{\st}{\mathrm{st}}

\newcommand{\lip}{\mathrm{Lip}}
\newcommand{\nU}{\mathcal{U}}

\newcommand{\vr}[2]{\mathrm{VR}(#1;#2)}
\newcommand{\vrleq}[2]{\mathrm{VR}_\leq(#1;#2)}
\newcommand{\vrless}[2]{\mathrm{VR}_<(#1;#2)}
\newcommand{\vrm}[2]{\mathrm{VR}^m(#1;#2)}
\newcommand{\vrmleq}[2]{\mathrm{VR}^m_\leq(#1;#2)}
\newcommand{\vrmless}[2]{\mathrm{VR}^m_<(#1;#2)}
\newcommand{\cechleq}[2]{\mathrm{\check{C}_\leq}(#1;#2)}
\newcommand{\cechless}[2]{\mathrm{\check{C}_<}(#1;#2)}
\newcommand{\cech}[2]{\mathrm{\check{C}}(#1;#2)}
\newcommand{\cechm}[2]{\mathrm{\check{C}}^m(#1;#2)}

\newcommand{\cechmless}[2]{\mathrm{\check{C}}^m_<(#1;#2)}
\newcommand{\id}{\mathrm{id}}
\newcommand{\so}{\mathrm{SO}}
\newcommand{\gh}{\mathrm{GH}}
\newcommand{\h}{\mathrm{H}}
\newcommand{\ph}{\mathrm{PH}}

\begin{document}

\title{Metric reconstruction via optimal transport}
\author{Micha{\l} Adamaszek}\email{aszek@mimuw.edu.pl}
\address{Department of Mathematics, University of Copenhagen, Universitetsparken 5,2100 Copenhagen, Denmark}
\author{Henry Adams}\email{adams@math.colostate.edu}
\address{Department of Mathematics, Colorado State University, Fort Collins, CO 80523, United States}
\author{Florian Frick}\email{frick@cmu.edu}
\address{Department of Mathematical Sciences, Carnegie Mellon University, Pittsburgh, PA 15213, USA}
\thanks{MA was supported by VILLUM FONDEN through the network for Experimental Mathematics in Number Theory, Operator Algebras, and Topology. The research of MA and HA was supported through the program ``Research in Pairs" by the Mathematisches Forschungsinstitut Oberwolfach in 2015.}
\keywords{Metric thickening, Vietoris--Rips complexes, Wasserstein metric, Karcher mean, Homotopy type}
\subjclass[2010]{
53C23, 
54E35, 
55P10, 
55U10
 }

\begin{abstract}
\small
Given a sample of points $X$ in a metric space $M$ and a scale~$r>0$, the Vietoris--Rips simplicial complex~$\vr{X}{r}$ is a standard construction to attempt to recover $M$ from $X$ up to homotopy type. A deficiency of this approach is that $\vr{X}{r}$ is not metrizable if it is not locally finite, and thus does not recover metric information about~$M$. We attempt to remedy this shortcoming by defining a metric space thickening of~$X$, which we call the \emph{Vietoris--Rips thickening}~$\vrm{X}{r}$, via the theory of optimal transport. When $M$ is a complete Riemannian manifold, or alternatively a compact Hadamard space, we show that the the Vietoris--Rips thickening satisfies Hausmann's theorem ($\vrm{M}{r}\simeq M$ for $r$ sufficiently small) with a simpler proof: homotopy equivalence $\vrm{M}{r}\to M$ is canonically defined as a center of mass map, and its homotopy inverse is the (now continuous) inclusion map $M\hookrightarrow\vrm{M}{r}$. Furthermore, we describe the homotopy type of the Vietoris--Rips thickening of the $n$-sphere at the first positive scale parameter $r$ where the homotopy type changes.
\end{abstract}

\maketitle

\section{Introduction}

Let $X$ be a set of points sampled from a metric space $M$; the only information we retain about $M$ is its metric restricted to~${X\times X}$. In general it is impossible to reconstruct the homotopy type of $M$ from this data, even if the set $X$ is sufficiently dense (say an $\varepsilon$-net). A remarkable theorem of Latschev~\cite{Latschev2001} states that if $M$ is a closed Riemannian manifold and $X$ is sufficiently close to $M$ in the Gromov--Hausdorff distance, then one can recover the homotopy type of~$M$. Indeed, consider the Vietoris--Rips complex~$\vr{X}{r}$, which has as its simplices the finite subsets of $X$ of diameter less than~${r>0}$. Latschev proves that for certain small values of~$r$, complex $\vr{X}{r}$ is homotopy equivalent to~$M$. The Vietoris--Rips complex was introduced independently by Vietoris who defined a homology theory for compact metric spaces (\cite{Vietoris27} and~\cite[VII \S5]{lefschetz1942algebraic}), and by Rips who showed that torsion-free hyperbolic groups have finite Eilenberg--MacLane spaces~\cite[Theorem~III 3.21]{bridson2011metric}.

More recently, Vietoris--Rips complexes have been a commonly used tool in computational topology and persistent homology~\cite{EdelsbrunnerHarer}. If $M$ is a manifold, what properties of $M$ can one recover when given only a finite data set $X$ noisily sampled from $M$? Latschev's theorem motivates the construction of the Vietoris--Rips complex $\vr{X}{r}$ as a proxy for the homotopy type of $M$, but unfortunately one does not know how to choose $r$ appropriately without knowledge of the (unknown) curvature of $M$. In practice, computational topologists instead let the scale $r$ vary from small to large, and compute the \emph{persistent homology} of $\vr{X}{r}$ (i.e.\ the homology of $\vr{X}{r}$ as $r$ changes) to get a multiresolution summary of the data. The topological perspective has aided the analysis of data arising from image processing~\cite{carlsson2008local}, conformation spaces of molecules~\cite{martin2010topology,zomorodian2012topological}, branched polymers~\cite{macpherson2012measuring}, and sensor networks~\cite{de2007coverage}, for example. Helpful expository introductions to persistent homology include~\cite{Carlsson2009} and~\cite{ghrist2008barcodes}.

Though we started with a metric space~$X$, the Vietoris--Rips complex $\vr{X}{r}$ does not come equipped with a natural choice of metric. Indeed, $\vr{X}{r}$ is not metrizable if it is not locally finite, meaning it is impossible to equip $\vr{X}{r}$ with a metric without changing the homeomorphism type. We use optimal transport to build a family of metric spaces $\vrm{X}{r}$ -- the \emph{Vietoris--Rips thickenings} -- from the knowledge of pairwise distances in~$X$. In these metric spaces we can take abstract convex combinations of points in $X$ whenever they are at distance less than~$r$ from one another, just as in $\vr{X}{r}$. If $\vr{X}{r}$ is not locally finite then necessarily $\vrm{X}{r}$ has a different (metrizable) topology. Furthermore, if $M$ is a closed Riemannian manifold and $X$ is Gromov--Hausdorff close to~$M$, then $\vrm{X}{r}$ is not only homotopy equivalent to~$M$ (for appropriate small~$r$), but also Gromov--Hausdorff close to~$M$ (Lemma~\ref{lem:thickening}). We prove the following (see Lemma~\ref{lem:thickening}, Corollary~\ref{cor:homeo}, Property~\ref{prop:weak-equiv}, Theorem~\ref{thm:metric-Hausmann}):

\begin{theorem-main}
Let $X$ be a metric space and $r>0$.
\begin{enumerate}
\item Metric space $\vrm{X}{r}$ is an $r$-thickening of~$X$; in particular the Gromov--Hausdorff distance between $X$ and $\vrm{X}{r}$ is at most~$r$.
\item If $\vr{X}{r}$ is locally finite, then $\vrm{X}{r}$ is homeomorphic to~$\vr{X}{r}$.
\item If $X$ is discrete, then $\vrm{X}{r}$ is homotopy equivalent to~$\vr{X}{r}$.
\item If $M$ is a complete Riemannian manifold with curvature bounded from above and below, then $\vrm{M}{r}$ is homotopy equivalent to $M$ for $r$ sufficiently small. 
\end{enumerate}
\end{theorem-main}

In the restricted setting where $X$ is discrete, item (1) is stated by Gromov~\cite[1.B(c)]{Gromov}. Similar properties are discussed when $X$ is a length space (and $\vrm{X}{r}$ is called a \emph{polyhedral regularization of $X$}) by Burago, Burago, Ivanov in~\cite[Example~3.2.9]{BuragoBuragoIvanov}. The proof of Item (3) relies on the nerve lemma, and a much more general statement is given in Remark~\ref{rem:nerve}. Item (4) is an analogue of Hausmann's theorem~\cite[Theorem~3.5]{Hausmann1995} for Vietoris--Rips thickenings, and holds also for compact Hadamard spaces (Remark~\ref{rem:Hausmann}). Whereas Hausmann's homotopy equivalence $\vr{M}{r}\to M$ relies on the choice of a total ordering of all points in~$M$, our homotopy equivalence $\vrm{M}{r}\to M$ is now canonically defined using Karcher or Fr{\'e}chet means. Furthermore, the homotopy inverse of our map is the (now continuous) inclusion $M\hookrightarrow\vrm{M}{r}$. We prove that the compositions are homotopy equivalent to the corresponding identity maps by using linear homotopies.

\subsection*{Motivation}

We provide the following motivation for our work. In applications of topology to data~\cite{Carlsson2009}, one is given a sampling of points $X$ from an unknown underlying space $M$ and would like to use $X$ to recover information of $M$. There are a variety of theoretical guarantees~\cite{AFV,AttaliLieutier,ChazalOudot2008,Latschev2001,niyogi2008finding} showing how Vietoris--Rips complexes and related constructions built on a sufficiently dense sampling $X$ can be used to recover information such as the homology groups and homotopy types of $M$. However, since $M$ is unknown one does not know if the assumptions needed for these results (such as having the Vietoris--Rips scale $r$ be sufficiently small depending on the curvature of $M$) are satisfied. In practical applications, one often allows the scale $r$ to vary from small to large and computes the persistent homology of $\vr{X}{r}$. This is reasonable because as $X$ converges to $M$ in the Gromov--Hausdorff distance, the persistent homology of $\vr{X}{r}$ converges to that of $\vr{M}{r}$~\cite{ChazalDeSilvaOudot2013}. However, very little is known about the limiting object, the persistent homology of $\vr{M}{r}$, even when $M$ is a manifold. Indeed, to our knowledge the only connected non-contractible manifold $M$ for which the persistent homology of $\vr{M}{r}$ is known at all scales $r$ is the circle~\cite{AA-VRS1}.\footnote{And easy consequences thereof, such as when manifold $M$ is an annulus or torus with a particular metric~\cite[Section~10]{AA-VRS1}.} When $M$ is an infinite metric space, we believe that the metric thickening $\vrm{M}{r}$ is in several ways a more natural object than the simplicial complex $\vr{M}{r}$. As evidence for this claim, in Section~\ref{sec:spheres} we describe the first new homotopy type of Vietoris--Rips thickenings $\vrm{S^n}{r}$ of higher-dimensional spheres as scale $r$ increases. We conjecture (Conjecture~\ref{conj:PH}) that these homotopy types are closely related to the (unknown) homotopy types of Vietoris--Rips complexes $\vr{S^n}{r}$ of higher-dimensional spheres, which determine what the persistent homology of a dataset $X$ will converge to when $X$ is a denser and denser sampling of a sphere.

We would like to clarify the relationship between using either the Vietoris--Rips complex $\vr{X}{r}$ or the Vietoris--Rips thickening $\vrm{X}{r}$ in applications of topology to data analysis. For $X$ finite, the persistent homology barcodes for $\vr{X}{r}$ and for $\vrm{X}{r}$ coincide (Corollary~\ref{cor:PH}). It is known that for $X$ infinite, the persistent homology intervals for $\vr{X}{r}$ and $\vrm{X}{r}$ can differ at endpoints (open endpoints versus closed endpoints, or vice-versa; see Remark~\ref{rem:PH}). It is not known whether the two persistent homology intervals are identical after ignoring endpoints, although we conjecture this to be the case (Conjecture~\ref{conj:PH}).\footnote{And in the unexpected event that the persistent homology intervals for $\vr{X}{r}$ and $\vrm{X}{r}$ can differ drastically, then it is not clear which one should be the primary object of interest.} When $X$ is infinite, it is difficult to determine the persistent homology of either $\vr{X}{r}$ or $\vrm{X}{r}$. A disadvantage of the thickening is that Latschev's theorem (\cite{Latschev2001}) is known only for $X$ finite (Corollary~\ref{cor:metric-Latschev}), although it is conjectured also for $X$ arbitrary (Conjecture~\ref{conj:metric-Latschev}). A second disadvantage of the thickening is that the stability of persistent homology (\cite[Theorem~5.2]{ChazalDeSilvaOudot2013}) is known only for $X$ and $Y$ finite (Corollary~\ref{cor:stability}), although it is conjectured also for $X$ and $Y$ arbitrary (Conjecture~\ref{conj:stability}). Advantages of the thickening over the complex are that the thickening is always a metric space, the proof of Hausmann's theorem is more natural, Hausmann's theorem extends also to compact Hadamard spaces and to Euclidean submanifolds~\cite{AM}, there are example spaces whose Vietoris--Rips complexes have uncountable homology but whose thickenings have finite homology (Appendix~\ref{app:Droz}), and furthermore we are able to determine the homotopy types of Vietoris--Rips thickenings of higher-dimensional spheres at larger scale parameters (Theorem~\ref{thm:S^n-critical}). It is not yet clear whether $\vr{X}{r}$ or $\vrm{X}{r}$ should be the primary object of study (despite our preferences for the latter), but understanding either space improves one's understanding of the other.

\subsection*{Organization}

In Section~\ref{sec:preliminaries} we introduce our primary tool, the Wasserstein or Kantorovich metric on Radon measures of an arbitrary metric space; see Edwards and Kellerer~\cite{edwards2011kantorovich,kellerer1984duality,kellerer1982duality}. We define Vietoris--Rips thickenings in Section~\ref{sec:vrm} and give some of their basic properties. Section~\ref{sec:Hausmann-nerve} proves metric analogues of Hausmann's theorem and the nerve lemma. In Section~\ref{sec:spheres} we leverage this new metric viewpoint to determine the homotopy types of~$\vrm{S^n}{r}$, for all spheres $S^n$, at the first positive scale parameter $r$ where the homotopy type changes. In Section~\ref{sec:maps} we study maps between simplicial complexes and thickenings, including metric analogues of Latschev's theorem and the stability of persistent homology.

Ivan Marin in~\cite{marin2017measure} studies related constructions which produce geometric classifying spaces for a topological group that is furthermore metrizable.

\section{Preliminaries}\label{sec:preliminaries}

\subsection*{Topological spaces}

We write $Y\simeq Z$ for homotopy equivalent topological spaces $Y$ and~$Z$. Given a topological space $Y$ and a subset $Z\subseteq Y$, let $\overline{Z}$ denote the closure of $Z$ in~$Y$. We denote the $n$-dimensional sphere by $S^n$ and the closed $n$-dimensional ball by $D^n$. Given a topological space~$Y$, let $C(Y)$ be its cone---for example $C(S^n)=D^{n+1}$. Furthermore, let $\Sigma\ Y$ be the suspension of~$Y$, and let $\Sigma^i\ Y$ be the $i$-fold suspension of~$Y$. The join of two topological spaces $Y$ and $Z$ is denoted~${Y*Z}$.

\subsection*{Metric spaces}

Given a metric space $(X,d)$, a point $x\in X$, and a real number $r\ge0$, we let $B_X(x,r)=\{y\in X~|~d(y,x)<r\}$ (or $B(x,r)$ when the ambient space~$X$ is clear from context) denote the open ball of radius $r$ about~$x$. We let $d_\gh(X,Y)$ denote the Gromov-Hausdorff distance between metric spaces $X$ and $Y$. 

An \emph{$r$-thickening} of a metric space $X$, as defined in~\cite[1.B]{Gromov}, is a larger metric space $Z\supseteq X$ such that
\begin{itemize}
\item the distance function on $Z$ extends that on $X$, and
\item $d(z,X)\le r$ for all $z\in Z$.
\end{itemize}
If $Z$ is an $r$-thickening of $X$ it follows that $d_\gh(X,Z)\le r$.

\subsection*{Simplicial complexes}

Let $K$ be a simplicial complex; we do not notationally distinguish between an abstract simplicial complex and its geometric realization. We let $V(K)$ denote the vertex set of $K$. If $V'\subseteq V(K)$, then we let $K[V']$ denote the induced simplicial complex on vertex set $V'$.

\subsection*{Vietoris--Rips and \v Cech complexes}

Let $X$ be a metric space and let $r\ge 0$. The Vietoris--Rips simplicial complex $\vrless{X}{r}$ (resp.\ $\vrleq{X}{r}$) has $X$ as its vertex set, and $\{x_0,\ldots,x_k\}\subseteq X$ as a simplex whenever $\diam(\{x_0,\ldots,x_k\})<r$ (resp.\ $\diam(\{x_0,\ldots,x_k\})\le r$). The \v Cech simplicial complex $\cechless{X}{r}$ (resp.\ $\cechleq{X}{r}$) has $X$ as its vertex set, and $\{x_0,\ldots,x_k\}\subseteq X$ as a simplex whenever $\bigcap_{i=0}^k B_X(x_i,r)\neq\emptyset$ (resp.\ $\bigcap_{i=0}^k \overline{B_X(x_i,r)}\neq\emptyset$). We write $\vr{X}{r}$ or $\cech{X}{r}$ when a statement is true for either choice of inequality, $<$ or $\leq$, applied consistently throughout.

\subsection*{The Wasserstein or Kantorovich metric}\label{ss:Wasserstein}

All of the statements in this subsection follow from~\cite{edwards2011kantorovich,kellerer1984duality,kellerer1982duality}, and we mainly use the notation from~\cite{edwards2011kantorovich}. Let $(X,d)$ be an arbitrary metric space. A measure $\mu$ defined on the Borel sets of $X$ is a \emph{Radon measure} if it is inner regular, i.e.\ $\mu(B)=\sup\{\mu(K)~|~K\subseteq B\mbox{ is compact}\}$ for all Borel sets~$B$, and if it is locally finite, i.e.\ every point $x\in X$ has a neighborhood $U$ such that $\mu(U)<\infty$. Let $\cP(X)$ denote the set of probability Radon measures such that for some (and hence all) $y\in X$, we have $\int_X d(x,y)\ d\mu<\infty$.

Define a metric on $X\times X$ by setting the distance between $(x_1,x_2),(x'_1,x'_2)\in X\times X$ to be $d(x_1,x'_1)+d(x_2,x'_2)$. Given $\mu,\nu\in\cP(X)$, let $\Pi(\mu,\nu)\subseteq\cP(X\times X)$ be the set of all probability Radon measures $\pi$ on $X\times X$ such that $\mu(E)=\pi(E\times X)$ and $\nu(E)=\pi(X\times E)$ for all Borel subsets $E\subseteq X$. 
\begin{definition}\label{def:Wasserstein}
The 1-Wasserstein metric on $\cP(X)$ is defined by
\[ d_{\cP(X)}(\mu,\nu)=\inf_{\pi\in\Pi(\mu,\nu)}\int_{X\times X}d(x,y)\ d\pi. \]
In what follows we denote the metric $d_{\cP(X)}\colon\cP(X)\times\cP(X)\to\R$ by $d\colon\cP(X)\times\cP(X)\to\R$, as it extends the metric $d\colon X\times X\to\R$.
\end{definition}

The infimum in this definition is attained (see~\cite[page~388]{edwards2011kantorovich}, and also~\cite{kellerer1984duality,kellerer1982duality}). This metric, which gives a solution to the Monge-Kantorovich problem, has many names: the Kantorovich, Wasserstein, optimal transport, or earth mover's metric; see Vershik's survey~\cite{vershik2013long}.

A generalization of the Kantorovich--Rubinstein theorem for arbitrary (possibly non-compact) metric spaces, proven in~\cite{edwards2011kantorovich,kellerer1984duality,kellerer1982duality}, states that the 1-Wasserstein metric satisfies
\begin{equation}\label{eq:dual}
d(\mu,\nu)=\sup\Biggl\{\int_X f(x)\ d(\mu-\nu)~\Big|~f\colon X\to\R,\ \lip(f)\le1\Biggr\}.
\end{equation}
Here $\lip(f)$ denotes the Lipschitz constant of~$f$. As a consequence, Definition~\ref{def:Wasserstein} indeed defines a metric on~$\cP(X)$.

Given a point $x\in X$, let $\delta_x\in\cP(X)$ be the Dirac probability measure with mass one at $x$. The map $X\to\cP(X)$ defined by $x\mapsto\delta_x$ is an isometry onto its image: for all points $x,y\in X$, we have
\begin{equation}\label{eq:isometry}
d(\delta_x,\delta_y)=d(x,y).
\end{equation}
In~\cite[page~99]{villani2008optimal}, Villani writes ``Wasserstein distances incorporate a lot of the geometry of the space. For instance, the mapping $x\mapsto\delta_x$ is an \emph{isometric} embedding \ldots but there are much deeper links." Our Theorems~\ref{thm:metric-Hausmann} and \ref{thm:metric-nerve} are evidence that the Wasserstein distances, at small scales, incorporate not only geometry but also homotopy information.

\section{The Vietoris--Rips thickening}\label{sec:vrm}

Consider points $x_0,\ldots,x_k \in X$. Whenever we write $\mu=\sum_{i=0}^k \lambda_i \delta_{x_i}$, we assume that $\lambda_i \ge 0$ for all $i$ and $\sum_i \lambda_i = 1$; hence $\mu$ is a probability measure. If we require each $\lambda_i>0$, then this representation is unique. Measure $\mu$ is locally finite (as a probability measure it is in fact finite), and it follows that $\mu$ is a Radon measure. Furthermore, note that $\mu\in\cP(X)$ since for any $y\in X$, we have $\int_X d(x,y)\ d\mu=\sum_{i=0}^k\lambda_id(x_i,y)<\infty$.

\begin{definition}\label{def:vrm}
Let $X$ be a metric space and $r\ge0$. The {\em Vietoris--Rips thickening} is the following submetric space of $\cP(X)$, equipped with the restriction of the 1-Wasserstein metric:
\begin{align*}
\vrmless{X}{r}&=\Biggl\{\sum_{i=0}^k \lambda_i\delta_{x_i}\in \cP(X)~\Big|~k\ge1,\ \lambda_i\ge0,\ \sum\lambda_i=1,\ \diam(\{x_0,\ldots,x_k\})< r\Biggr\}\quad\mbox{or}\\
\vrmleq{X}{r}&=\Biggl\{\sum_{i=0}^k \lambda_i\delta_{x_i}\in \cP(X)~\Big|~k\ge1,\ \lambda_i\ge0,\ \sum\lambda_i=1,\ \diam(\{x_0,\ldots,x_k\})\leq r\Biggr\}.
\end{align*}
\end{definition}
\noindent We include a superscript $m$ to denote ``metric." By convention $\vrmless{X}{0}$ is the empty set, and $\vrmleq{X}{0}$ is equal to $X$ as a metric space.

Thus the Vietoris--Rips thickening of a metric space~$X$ is the space of all (abstract) convex combinations of nearby points in~$X$ with the Wasserstein metric. A qualitative difference to the usual Vietoris--Rips complex is that the natural embedding $X \to \vrm{X}{r}$ is now a continuous map (and even an isometric embedding). In fact, we will naturally consider $X$ to be a subset of $\vrm{X}{r}$ and write $\sum_{i=0}^k \lambda_ix_i$ for a point in $\vrm{X}{r}$ instead of $\sum_{i=0}^k \lambda_i\delta_{x_i}$.

Given two points $\mu=\sum_{i=0}^k \lambda_ix_i$ and $\mu'=\sum_{j=0}^{k'} \lambda'_jx'_j$ in $\vrm{X}{r}$, note that an element $\pi\in\Pi(x,x')$ can be written as
\[ \pi=\sum_{\substack{1\le i\le k\\1\le j\le k'}} \pi_{i,j}\delta_{x_i,x'_j}\quad\mbox{with}\quad \pi_{i,j}\ge0,\quad\sum_j \pi_{i,j}=\lambda_i,\quad\mbox{and}\quad\sum_i \pi_{i,j}=\lambda'_j. \]
The cost of a matching in this finite setting is $\int_{X\times X}d(x,y)\ d\pi = \sum_{i,j} \pi_{i,j}d(x_i,x'_j)$.
One can analogously define a $p$-Wasserstein metric on $\vrm{X}{r}$ for any $1\le p\le\infty$.

\begin{remark}
The Vietoris--Rips thickening $\vrm{X}{r}$ need not be homeomorphic to its corresponding complex $\vr{X}{r}$. Indeed, as a simplicial complex $\vr{X}{r}$ is metrizable if and only if it is locally finite; see Sakai~\cite[Proposition~4.2.16(2)]{sakai2013geometric}. In other words, if $\vr{X}{r}$ is not locally finite, then it is impossible to equip it with any metric without changing the homeomorphism type.
\end{remark}

\begin{remark}
The Vietoris--Rips thickening $\vrmleq{X}{r}$ need not be homotopy equivalent to~$\vrleq{X}{r}$. Indeed, when $r=0$ note $\vrmleq{X}{0}$ is equal to $X$ as a metric space, whereas $\vrleq{X}{0}$ is equipped with the discrete topology. A less trivial example is that $\vrmleq{S^1}{\frac{1}{3}}\simeq S^3$ (Remark~\ref{rem:S1}), whereas $\vrleq{S^1}{\frac{1}{3}}\simeq\bigvee^{\mathfrak{c}}S^2$; see~\cite{AA-VRS1}. Appendix~\ref{app:Droz} gives other examples where $\vrleq{X}{r}$ has uncountable homology but $\vrmleq{X}{r}$ does not. We leave it as an open question whether with the $<$ convention $\vrmless{X}{r}$ and $\vrless{X}{r}$ are homotopy equivalent, though this is known to be true in certain cases (Remark~\ref{rem:nerve}, Property~\ref{prop:weak-equiv}).
\end{remark}

\begin{remark}
The metric thickening $\vrm{X}{r}$ is rarely complete or compact, except when $X$ is finite or discrete. Indeed, see Section~\ref{sec:not-compact}. This situation could be improved by considering a different space, those measures $\mu$ whose support has diameter at most $r$, with no requirement that the support be finite. Theorem~6.18 of~\cite{villani2008optimal} implies that if $X$ is Polish (complete and separable), then this space of measures of bounded support is Polish (indeed, it is a closed subset of a Polish space). Similarly,~\cite[Remark~6.19]{villani2008optimal} implies that this space of measures of bounded support is compact if $X$ is. Furthermore, for $X$ compact a metric version of Hausmann's theorem would still be true for this space of measures of bounded support, as outlined in Remark~\ref{rem:Hausmann}. We instead consider the Vietoris--Rips thickening $\vrm{X}{r}$ in which all measures have finite support, as it is more directly related to the Vietoris--Rips complex (in which all simplices are finite). Nevertheless, the space of measures of bounded support (without an assumption of finiteness) is natural to consider. 
\end{remark}

The Vietoris--Rips thickening is only one example of a more general construction.

\begin{definition}\label{def:complex-thickening}
Let $X$ be a metric space and let $K$ be a simplicial complex with vertex set $V(K)=X$. The \emph{metric thickening $K^m$} is the following submetric space of $\cP(X)$, equipped with the restriction of the 1-Wasserstein metric:
\[ K^m=\Biggl\{\sum_{i=0}^k \lambda_i\delta_{x_i}\in \cP(X)~\Big|~k\ge1,\ \lambda_i\ge0,\ \sum\lambda_i=1,\ \{x_0,\ldots,x_k\}\in K\Biggr\}. \]
\end{definition}

Examples include not only Vietoris--Rips thickenings $\vrm{X}{r}=(\vr{X}{r})^m$, but also \v Cech thickenings $\cechm{X}{r}=(\cech{X}{r})^m$, alpha thickenings, and witness thickenings, etc. Alpha simplicial complexes are defined for example in~\cite{EdelsbrunnerHarer}, and witness complexes in~\cite{DeSilvaCarlsson,ChazalDeSilvaOudot2013}. As before, we write $\sum_{i=0}^k \lambda_ix_i\in K^m$ instead of $\sum_{i=0}^k \lambda_i\delta_{x_i}$.

The following lemma  is stated by Gromov (\cite[1.B(c)]{Gromov}) in the specific case where $X$ is discrete.

\begin{lemma}\label{lem:thickening}
If $K^m$ is a metric thickening of metric space $X$ such that each simplex of $K$ has diameter at most $r\ge0$, then $K^m$ is an $r$-thickening of~$X$.
\end{lemma}

\begin{proof}
It is clear by \eqref{eq:isometry} that the Wasserstein metric on $K^m$ extends the metric on $X$. Next, consider any point $\sum_i \lambda_ix_i$ in~$K^m$. Note
\[ d\Bigl(\sum_i \lambda_ix_i,X\Bigr)\le d\Bigl(\sum_i \lambda_ix_i,x_0\Bigr)=\sum_i \lambda_i d(x_i,x_0)\le r\sum_i \lambda_i=r.\]
\end{proof}

We may therefore say $K^m$ is a \emph{metric $r$-thickening} if each simplex of $K$ has diameter at most $r$. Note $\vrm{X}{r}$ is a metric $r$-thickening and $\cechm{X}{r}$ is a metric $2r$-thickening.

We briefly relate metric thickenings to configuration spaces. Let $X$ be a metric space and $K$ a simplicial complex on vertex set $X$. Denote the \emph{$n$-skeleton} of~$K^m$ by $S_n(K^m) = \{\sum_{i=0}^k \lambda_i x_i \in K^m~|~ k\le n\}$, and denote the configuration space of $n+1$ unordered points in $X$ by~$C_{n+1}X$. If $K$ is the maximal simplicial complex on vertex set~$X$, then we have $S_n(K^m)\setminus S_{n-1}(K^m)\simeq C_{n+1}(X)$ where a deformation retraction is obtained by collapsing each open simplex to its barycenter. A related space is $\exp_{n+1} X$, the set of all finite subsets of $X$ of cardinality at most $n+1$, as studied by Tuffley in~\cite{tuffley2002finite}.

\subsection{Basic properties}

We describe a basic result on the continuity of maps between metric thickenings induced from maps on the underlying metric spaces, and a result on the continuity of linear homotopies on metric thickenings.

Let $X$, $Y$ be metric spaces and $K$, $L$ be simplicial complexes with vertex sets $V(K)=X$, $V(L)=Y$. Note that if $f\colon X\to Y$ has the property that $f(\sigma)$ is a simplex in $L$ for each simplex $\sigma\in K$, then the induced map $\tilde{f} \colon K^m \to L^m$ defined by $\sum_i \lambda_ix_i \mapsto \sum_i \lambda_if(x_i)$ exists. For example, if $f\colon X\to Y$ has the property that $d(x,x')\le r$ implies $d(f(x),f(x'))\le r$, then the induced map $\tilde{f} \colon \vrmleq{X}{r} \to \vrmleq{Y}{r}$ exists.

\begin{lemma}\label{lem:map-lipschitz}
Let $X$, $Y$ be metric spaces and $K$, $L$ be simplicial complexes with vertex sets $V(K)=X$, $V(L)=Y$. Let $f\colon X\to Y$ be a map of metric spaces such that the induced map $\tilde{f} \colon K^m \to L^m$ on metric thickenings exists. If $f$ is $c$-Lipschitz, then so is~$\tilde{f}$.
\end{lemma}

\begin{proof}
Let $\varepsilon>0$. Since $f$ is $c$-Lipschitz, we have $d(f(x),f(x'))\le c d(x,x')$ for all $x,x'\in X$. Now suppose $d(\sum \lambda_ix_i,\sum \lambda'_jx'_j)\leq\frac{\varepsilon}{c}$. This means there is some $\pi_{i,j}\ge0$ with $\sum \pi_{i,j}=1$, $\sum_j \pi_{i,j}=\lambda_i$, $\sum_i \pi_{i,j}=\lambda'_j$, and $\sum \pi_{i,j}d(x_i,x'_j)\leq\frac{\varepsilon}{c}$. It follows that
\begin{align*}d\bigl(\tilde{f}(\sum \lambda_ix_i),\tilde{f}(\sum \lambda'_jx'_j)\bigr)=d\bigl(\sum \lambda_if(x_i),\sum \lambda'_jf(x'_j)\bigr)&\le\sum_{i,j} \pi_{i,j}d(f(x_i),f(x'_j))\\
&\le c\sum_{i,j}\pi_{i,j} d(x_i,x'_j)\leq\varepsilon,
\end{align*}
and hence $\tilde{f}$ is $c$-Lipschitz.
\end{proof}

\begin{remark}\label{rem:not-cont}
Note that $f$ continuous need not imply that $\tilde{f}$ is continuous. Indeed, consider
\[ X=\{(0,0)\}\cup\{(\tfrac{1}{n},0)~|~n\in\N\}\cup\{(\tfrac{1}{n},1)~|~n\in\N\}\subseteq\R^2. \]
Define the continuous map $f\colon X\to\R$ by
\[ f(x,y)=\begin{cases}
0&\mbox{if }y=0\\
n^2&\mbox{if }(x,y)=(\frac{1}{n},1).
\end{cases} \]
Let $K=\vr{X}{\infty}$ and $L=\vr{\R}{\infty}$ be the maximal simplicial complexes on their vertex sets. Note that $\mu_n=\frac{n-1}{n}\delta_{(\frac{1}{n},0)}+\frac{1}{n}\delta_{(\frac{1}{n},1)}$ is a sequence in $K^m$ converging to $\delta_{(0,0)}$, but that $\tilde{f}(\mu_n)=\frac{n-1}{n}\delta_{0}+\frac{1}{n}\delta_{n^2}$ is not Cauchy (and hence not convergent) in~$L^m$.
\end{remark}

A metric thickening $K^m$ abstractly is the space of convex combinations of points in~$X$, and taking convex combinations (wherever defined) is a continuous operation in the Wasserstein metric. This is made rigorous in the following lemma.

\begin{lemma}\label{lem:homotopy-cont}
Suppose $K^m$ is a metric $r$-thickening and $f\colon K^m\to K^m$ is a continuous map such that $H\colon K^m\times [0,1]\to K^m$ given by $H(\mu,t)=(1-t)\mu+tf(\mu)$ is well-defined. Then $H$ is continuous.
\end{lemma}

\begin{proof}
Note that $H$ is the composition of the continuous map $(\mu,t)\mapsto(\mu,f(\mu),t)$ with $L\colon K^m\times K^m\times I\to K^m$ defined by $L(\mu,\nu,t)=(1-t)\mu+t\nu$. For any map $g\colon X\to \R$ with Lipschitz constant at most one, and for any $(\mu,\nu,t),(\mu',\nu',t')\in K^m\times K^m\times I$, we have
\begin{align*}
&\int_X g(x)\ d((1-t)\mu+t\nu-(1-t')\mu'-t'\nu')\\
=&(1-t)\int_X g(x)\ d(\mu-\mu') + t\int_X g(x)\ d(\nu-\nu') + (t'-t)\int_X g(x)\ d(\mu'-\nu')\\
\le&\max\{d(\mu,\mu'),d(\nu,\nu')\} + |t'-t|d(\mu',\nu').
\end{align*}
Hence by \eqref{eq:dual} we have $d(L(t,\mu,\nu),L(t',\mu',\nu'))\le\max\{d(\mu,\mu'),d(\nu,\nu')\} + |t'-t|d(\mu',\nu')$, and so $L$ is continuous. It follows that $H$ is continuous.
\end{proof}

\section{Metric analogues of Hausmann's theorem and the nerve lemma}\label{sec:Hausmann-nerve}

In this section we prove metric analogues of Hausmann's theorem, the nerve lemma, and Latschev's theorem. Let $M$ be a Riemannian manifold. Hausmann's theorem~\cite[Theorem~3.5]{Hausmann1995} states that there exists some real number $r(M)$ sufficiently small (depending on the scalar curvature of~$M$) such that if $0<r\le r(M)$, then $\vrless{M}{r}$ is homotopy equivalent to~$M$. Hausmann's homotopy equivalence $T\colon\vrless{M}{r}\to M$ is extremely non-canonical; it depends on the choice of an arbitrary total ordering of all of the points in~$M$. Furthermore, if $M$ is not a discrete metric space then the inclusion $M\hookrightarrow\vrless{M}{r}$ is not continuous and hence cannot be a homotopy inverse. By contrast, we use Karcher means to define a map $g\colon\vrm{M}{r}\to M$ which is a homotopy equivalence for $r$ sufficiently small (Theorem~\ref{thm:metric-Hausmann}). A key feature of our proof is that we give a canonical choice for map $g$, which furthermore has the continuous inclusion $M\hookrightarrow\vrm{M}{r}$ as a homotopy inverse. A similar proof works also for the \v{C}ech thickening in Theorem~\ref{thm:metric-nerve}.

The following technology regarding Karcher means is from~\cite{Karcher1977}. Let the complete Riemannian manifold $M$ and real number $\rho>0$ satisfy the following conditions:
\begin{enumerate}
\item[(i)] For each $m \in M$, the geodesic ball $B(m,\rho)$ of radius $\rho$ about $m$ is convex, meaning the shortest geodesic between any two points in $B(m,\rho)$ is unique in $M$ and lies in~$B(m,\rho)$.
\item[(ii)] The manifold $M$ has sectional curvature bounds $\delta\le K\le\Delta$, where if $\Delta>0$ then we also assume $2\rho<\frac{1}{2}\pi\Delta^{-1/2}$.
\end{enumerate}
Given a probability measure $\mu$ with support contained in an open ball $B$ of $M$ of radius at most $\rho$, the function $P_\mu\colon\overline{B}\to\R$ defined by
\[ P_\mu(x)=\frac{1}{2}\int_{y\in M} d(x,y)^2\ d\mu \]
has a unique minimum in $B$ (see~\cite[Definition~1.3]{Karcher1977}). This minimizer is denoted $C_\mu$, and called the \emph{center of mass} or the \emph{Karcher mean}.
Furthermore,~\cite[Corollary~1.6]{Karcher1977} bounds the variation of the Karcher mean: if $\mu$ and $\nu$ are probability measures with support in an open ball $B$ of radius $\rho$, then
\begin{equation}\label{eq:cont}
d(C_{\mu},C_{\nu})\le(1+c(\delta,\Delta)(2\rho)^2)\int_{(x,y)\in M\times M} d(x,y)\ d(\mu\times\nu).
\end{equation}

Let $K$ a simplicial complex on vertex set $M$ such that the vertices of each simplex are contained in an open ball of radius $\rho$. Note that a point $\mu=\sum_{i=0}^k \lambda_i x_i$ in the metric thickening $K^m$ is a measure with support contained in an open ball of radius at most $\rho$.
Using the Karcher mean, we define a map $g\colon K^m\to M$ by setting $g(\mu)=C_\mu$.

\begin{lemma}\label{lem:g-cont}

Let $M$ be a complete Riemannian manifold and $\rho>0$ be a real number satisfying (i) and (ii). Suppose $K^m$ is a metric thickening of $M$ such that each simplex of $K$ has diameter at most $r<\rho$. Then the map $g\colon K^m\to M$ is continuous.
\end{lemma}

\begin{proof}
Let $\mu=\sum_{i=0}^k \lambda_i x_i$ and $\mu'=\sum_{j=0}^{k'} \lambda'_j x'_j$. By the definition of $d(\mu,\mu')$ there exists some $\pi_{i,j}\ge0$ with $\sum_j \pi_{i,j}=\lambda_i$, $\sum_i \pi_{i,j}=\lambda'_j$, and $\sum_{i,j}\pi_{i,j}d(x_i,x'_j)= d(\mu,\mu')$.
Note that if $d(\mu,\mu')<\rho-r$, then there is at least one pair of vertices $x_i\in\{x_0,\ldots,x_k\}$ and $x'_j\in\{x'_0,\ldots,x'_{k'}\}$ with $d(x_i,x'_j)<\rho-r$. By the triangle inequality we have
\[\{x_0,\ldots,x_k\}\cup\{x'_0,\ldots,x'_{k'}\}\subseteq B(x_i,\rho).\]
Hence for all $\mu,\mu'$ with $d(\mu,\mu')<\rho-r$, we have
\begin{align*}
d(g(\mu),g(\mu'))&=d(C_\mu,C_{\mu'})\le(1+c(\delta,\Delta)(2\rho)^2)\int_{M\times M} d(x,y)\ d(\mu\times\mu')&&\mbox{by \eqref{eq:cont}}\\
&=\bigl(1+c(\delta,\Delta)(2\rho)^2\bigr)\sum_{i,j}\pi_{i,j}d(x_i,x'_j)\le(1+c(\delta,\Delta)(2\rho)^2)d(\mu,\mu').
\end{align*}
So the map $g\colon K^m\to M$ is continuous.
\end{proof}

We can now prove our main result: the Vietoris--Rips thickening of a complete Riemannian manifold~$M$ accurately recovers the homotopy type of~$M$ for sufficiently small distance parameter~$r$.

\begin{theorem}[Metric Hausmann's theorem]\label{thm:metric-Hausmann}
Let $M$ be a complete Riemannian manifold and $\rho>0$ be a real number satisfying (i) and~(ii). For $r<\rho$, the map $g\colon\vrm{M}{r}\to M$ is a homotopy equivalence.
\end{theorem}

\begin{proof}
The inclusion map $\iota\colon M\to\vrm{M}{r}$ is an isometric embedding, and hence continuous. Since $r<\rho$, the map $g\colon\vrm{M}{r}\to M$ is defined and continuous. We will show that $\iota$ and  are homotopy inverses. Note $g\circ\iota=\id_{M}$.

We must show $\iota\circ g\simeq\id_{\vrm{M}{r}}$. Define $H\colon\vrm{M}{r}\times[0,1]\to\vrm{M}{r}$ by $H(\mu,t)=(1-t)\mu+t(\iota\circ g)(\mu)$. Note $H(-,0)=\id_{\vrm{M}{r}}$ and $H(-,1)=\iota\circ g$, and hence it suffices to show that $H$ is well-defined and continuous.

To see that $H$ is well-defined, let $\mu=\sum_{i=0}^k \lambda_i x_i\in\vrm{M}{r}$; it suffices to show that $[x_0,\ldots,x_k,g(\mu)]$ is a simplex in $\vr{M}{r}$. For the $<$ case, note $\mu\in\vrmless{M}{r}$ implies $\{x_0,\ldots,x_k\}\subseteq B(x_i,r)$ for all $i$, giving $g(\mu)\in\cap_i B(x_i,r)$ by~\cite[Definition~1.3]{Karcher1977} as required. For the $\le$ case, for each $i$ and $\varepsilon>0$ we have that $\{x_0,\ldots,x_k\}\subseteq B(x_i,r+\varepsilon)$, giving $g(\mu)\in\cap_i B(x_i,r+\varepsilon)$ for all $\varepsilon>0$ and hence $g(\mu)\in\cap_i \overline{B(x_i,r)}$. In either case we have that $[x_0,\ldots,x_k,g(\mu)]\in\vr{M}{r}$, so $H$ is well-defined.

Map $H$ is continuous by Lemma~\ref{lem:homotopy-cont}, and therefore $g\colon\vrm{M}{r}\to M$ is a homotopy equivalence.
\end{proof}

The paper~\cite{AM} proves a variant of the metric Hausmann's theorem when manifold $M$ is instead a subset of Euclidean space, equipped with the Euclidean metric.

\begin{remark}\label{rem:Hausmann}
The following observation is due to an anonymous referee. Generalize the definition of the Karcher mean to arbitrary metric spaces, and suppose that $X$ is a compact metric space and $K$ a simplicial complex with vertices in a subset of $X$, such that for any measure $\mu \in K^m$ the Karcher mean is well-defined, and, in particular, unique. One such example is when $X$ is a Hadamard space, i.e., a globally nonpositively curved space~\cite[Proposition~4.3]{sturm2003probability}. Then the map $g$ in Lemma~\ref{lem:g-cont} is well-defined and continuous. To see the continuity of $g$, let $\mu_n \in K^m$ be a sequence of measures converging to~$\mu$. Then by compactness of $X$ their Karcher means $g(\mu_n)$ have a convergent subsequence with limit point~$z$. Now,
\[\int_X |g(\mu_n) - x|^2\ d \mu_n(x) \le \int_X |y - x|^2\ d \mu_n(x)\]
for all $y$ by definition of the Karcher mean. In particular, by passing to the convergent subsequence and taking limits we see that $\int_X |z - x|^2\ d \mu(x) \le \int_X |y - x|^2\ d \mu(x)$ for all $y$. Thus $z$ is the Karcher mean of~$\mu$. Certainly the sequence $g(\mu_n)$ does not have an accumulation point different from~$z$, as this would contradict the uniqueness of Karcher means. By compactness $g(\mu_n)$ converges to $z = g(\mu)$. This shows the continuity of~$g$. Notice that all we used is that Karcher means of a converging sequence of measures $\mu_n \in K^m$ are eventually contained in a compact set. By the same proof as that of Theorem~\ref{thm:metric-Hausmann}, it follows that if $X$ is a compact metric space and $\vrm{X}{r}$ is a metric thickening such that for any measure $\mu \in \vrm{X}{r}$ the Karcher mean is well-defined, then the continuous map $g\colon \vrm{X}{r}\to X$ is a homotopy equivalence.
\end{remark}

Recall that the \v Cech complex $\cech{M}{r}$ has vertex set $M$ and a $k$-simplex $\{x_0, \dots, x_k\}$ if $\bigcap_{i=0}^k B_r(x_i) \ne \emptyset$. For $r$ sufficiently small all these $r$-balls and their nonempty intersections are contractible, and thus $\cech{M}{r}$ is homotopy equivalent to $M$ by the nerve lemma. The metric \v Cech thickening $\cechm{M}{r}$ is not a simplicial complex and the nerve lemma does not apply. Nonetheless, in a similar fashion to Theorem~\ref{thm:metric-Hausmann} we can show that our metric analogue of the \v Cech complex is homotopy equivalent to~$M$. 

\begin{theorem}[Metric nerve lemma]\label{thm:metric-nerve}
Let $M$ be a complete Riemannian manifold and $\rho>0$ be a real number satisfying (i) and~(ii). For $r<\frac{\rho}{2}$, the map $g\colon\cechm{M}{r}\to M$ is a homotopy equivalence.
\end{theorem}

\begin{proof}
We will show that the inclusion $\iota\colon M\to\cechm{M}{r}$ and the Karcher mean map $g\colon \cechm{M}{r} \to M$ are homotopy inverses. Note $g$ is defined and continuous by Lemma~\ref{lem:g-cont} since $r<\frac{\rho}{2}$. As before define $H\colon\cechm{M}{r}\times[0,1]\to\cechm{M}{r}$ by $H(\mu,t)=(1-t)\mu+tg(\mu)$. By Lemma~\ref{lem:homotopy-cont} it suffices to show that $(1-t)\mu+tg(\mu) \in \cechm{M}{r}$ for any $\mu \in \cechm{M}{r}$.  Let $\mu=\sum_{i=0}^k \lambda_i x_i\in\cechmless{M}{r}$, that is, there is a $y \in M$ with $d(y,x_i) < r$ for all $i$. Since $x_0, \dots, x_k \in B(y,r)$ we have that $g(\mu) \in B(y,r)$, which implies that $t\mu+(1-t)g(\mu) \in \cechmless{M}{r}$. The $\le$ case follows from the same approximation argument as in the proof of Theorem~\ref{thm:metric-Hausmann}.
\end{proof}

\section{Vietoris--Rips thickenings of spheres}\label{sec:spheres}

In this section we study the Vietoris--Rips thickenings of $n$-spheres, and in particular, in Theorem~\ref{thm:S^n-critical} we describe the first new homotopy type (after that of the $n$-sphere) that appears as the scale parameter increases. In order to prove this result, we first need to study the continuity properties of maps from metric thickenings into Euclidean space.

\subsection*{Maps to Euclidean space}

Let $X$ be a metric space and $K$ a simplicial complex with $V(K)=X$. We study when a map from $X$ into $\R^n$ induces a continuous map on a metric thickening $K^m$ of $X$. Given a function $f\colon X\to\R^n$, by an abuse of notation we also let $f\colon K^m\to\R^n$ denote the map defined by $\sum \lambda_ix_i\mapsto\sum \lambda_if(x_i)$, where $\sum \lambda_if(x_i)$ is a linear combination of vectors in $\R^n$.

\begin{remark}
Note that $f\colon X\to \R^n$ continuous need not imply $f\colon K^m\to\R^n$ is continuous. Indeed, as in Remark~\ref{rem:not-cont}, let
\[ X=\{(0,0)\}\cup\{(\tfrac{1}{n},0)~|~n\in\N\}\cup\{(\tfrac{1}{n},1)~|~n\in\N\}\subseteq\R^2. \]
Define the continuous map $f\colon X\to\R$ by $f(x,0)=0$ and $f(\frac{1}{n},1)=n^2$. Let $K=\vr{X}{\infty}$ be the maximal simplicial complex on vertex set $X$. Note that $\mu_n=\frac{n-1}{n}\delta_{(\frac{1}{n},0)}+\frac{1}{n}\delta_{(\frac{1}{n},1)}$ is a sequence in $K^m$ converging to $\delta_{(0,0)}$, but that $f(\mu_n)=\frac{1}{n}n^2=n$ is not a convergent sequence in~$\R$.
\end{remark}



The next lemma follows from~\cite[Theorem~6.9]{villani2008optimal} if $X$ is Polish, (or from~\cite[Theorem~5.11]{santambrogio2015optimal} if $X$ is Euclidean). We have included a proof since for finitely supported measures, the result also holds in the non-Polish case.

\begin{lemma}\label{lem:map-Rn-continuous-bounded}
Let $K^m$ be a metric thickening of metric space $X$. If $f\colon X\to \R^n$ is continuous and bounded, then so is $f \colon K^m \to \R^n$.
\end{lemma}

\begin{proof}
Let $C$ be such that $\|f(x)-f(y)\|\leq C$ for all $x,y\in X$. Fix a point $\sum \lambda_i x_i\in K^m$ and $\varepsilon>0$. Using the continuity of $f$ at the finitely many points $x_1,\ldots,x_n$, choose $\delta>0$ so that $d(x_i,y)\le\delta$ implies $\|f(x_i)-f(y)\|\leq \varepsilon/2$ for all $i$. Reducing $\delta$ if necessary, we can also assume $\delta\le\frac{\varepsilon}{2C}$. We will show that $d(\sum \lambda_i x_i,\sum \lambda'_j x_j')\le\delta^2$ implies $\|f(\sum \lambda_i x_i)-f(\sum \lambda'_j x_j')\|\le\varepsilon$, which proves the continuity of $f \colon K^m \to \R^n$ at $\sum \lambda_i x_i$.

Let $\pi_{i,j}$ be a matching from $\sum \lambda_i x_i$ to $\sum \lambda'_j x_j'$ with $\sum_{i,j} \pi_{i,j}d(x_i,x_j')\le\delta^2$. Let $A=\{(i,j)~|~d(x_i,x_j')\geq \delta\}$ and $B=\{(i,j)~|~d(x_i,x_j')< \delta\}$. We have
\[\delta\sum_A\pi_{i,j}\leq \sum_A \pi_{i,j}d(x_i,x_j')\leq \sum_{i,j}\pi_{i,j}d(x_i,x_j')\le\delta^2,\]
so $\sum_A \pi_{i,j}\le\delta$. Hence
\begin{align*}
\Bigl\|f(\sum \lambda_i x_i)-f(\sum \lambda'_j x_j')\Bigr\|&=\Bigl\| \sum_i\lambda_if(x_i)-\sum_j\lambda'_jf(x'_j) \Bigr\|=\Bigl\| \sum_{i,j}\pi_{i,j}f(x_i)-\sum_{i,j}\pi_{i,j}f(x'_j) \Bigr\|\\
&\le\sum_{i,j} \pi_{i,j}\|f(x_i)-f(x'_j)\|\\
&=\sum_A \pi_{i,j}\|f(x_i)-f(x'_j)\|+\sum_B \pi_{i,j}\|f(x_i)-f(x'_j)\|\\
&\le C\sum_A \pi_{i,j}+\frac{\varepsilon}{2}\sum_B \pi_{i,j}\le C\delta+\varepsilon/2\le\varepsilon.
\end{align*}

To see that $f \colon K^m \to \R^n$ is bounded, note that $f(K^m)$ is contained in $\conv(f(X))$.
\end{proof}

A similar result, which is not hard to verify, is that if $K^m$ is a metric thickening of metric space $X$, and if $f\colon X\to \R^n$ is $c$-Lipschitz, then so is $f \colon K^m \to \R^n$.

\subsection*{Vietoris--Rips thickenings of spheres}

Let $A_{n+2}$ be the alternating group on $n+2$ elements. For example, the group $A_3$ is isomorphic to $\Z/3\Z$, and the group $A_4$ is also known as the tetrahedral group. In a non-canonical fashion, we can view $A_{n+2}$ as a subgroup of $\so(n+1)$, as follows. Fix a regular $(n+1)$-dimensional simplex inscribed in $S^n$ inside $\R^{n+1}$, with the center of the simplex at the origin. The $(n+1)$-simplex has $n+2$ vertices, and $A_{n+2}$ as its group of rotational symmetries. We can therefore associate each element $g\in A_{n+2}$ with a rotation matrix in $\so(n+1)$ that permutes the vertices of the simplex in the same way that $g$ does. After this non-canonical identification as a subgroup, $A_{n+2}$ acts on $\so(n+1)$ via left multiplication, and we let $\tfrac{\so(n+1)}{A_{n+2}}$ be the orbit space of this action.\footnote{By~\cite[Section~1.3, Exercise 24(b)]{Hatcher}, the homeomorphism type of $\tfrac{\so(n+1)}{A_{n+2}}$ is unchanged if one identifies $A_{n+2}$ as a subgroup of $\so(n+1)$ using a different inscribed $(n+1)$-simplex.} We also think of $\tfrac{\so(n+1)}{A_{n+2}}$ as the moduli space of regular $(n+1)$-simplices inscribed in $S^n\subseteq\R^{n+1}$.

Denote by $S^n$ the $n$-dimensional sphere equipped with either the Euclidean or the geodesic metric, and let $r_n$ be the diameter of an inscribed regular $(n+1)$-simplex in $S^n$. We will show in Proposition~\ref{prop:S^n-small} and Theorem~\ref{thm:S^n-critical} that $\vrm{S^n}{r}\simeq S^n$ for $0<r<r_n$, that $\vrmless{S^n}{r_n}\simeq S^n$, and that $\vrmleq{S^n}{r_n}\simeq\Sigma^{n+1}\ \tfrac{\so(n+1)}{A_{n+2}}$.

In particular, let $S^1$ be the circle of unit circumference equipped with the path-length metric; this gives $r_1=\frac{1}{3}$. In~\cite{AA-VRS1} the first two authors show that the Vietoris--Rips simplicial complexes of the circle satisfy $\vr{S^1}{r}\simeq S^1$ for $0<r<\frac{1}{3}$, that $\vrless{S^1}{\frac{1}{3}}\simeq S^1$, that $\vrleq{S^1}{\frac{1}{3}}\simeq\bigvee^{\mathfrak{c}}S^2$, and that $\vr{S^1}{r}\simeq S^3$ for $\frac{1}{3}<r<\frac{2}{5}$. By contrast, in the case of metric thickenings Remark~\ref{rem:S1} gives that
\[ \vrmleq{S^1}{\tfrac{1}{3}}\simeq\Sigma^2\ \tfrac{\so(2)}{A_3}=\Sigma^2\ S^1=S^3. \]

Let $f \colon \vrmleq{S^n}{r} \to \R^{n+1}$ be the projection map sending a finite convex combination of points in $S^n$ to its corresponding linear combination in $\R^{n+1}$. This map $f$ is continuous by Lemma~\ref{lem:map-Rn-continuous-bounded}. Let $\pi \colon \R^{n+1}\setminus\{\vec{0}\} \to S^n$ be the radial projection map. In addition, let $W$ be the set of all interior points of regular $(n+1)$-simplices inscribed in $\vrleq{S^n}{r_n}$. More precisely,
\[ W = \bigl\{\textstyle{\sum_{i=0}^{n+1}}\lambda_i x_i~\big|~\lambda_i>0\mbox{ for all }i\mbox{ and }\{x_0,\ldots,x_{n+1}\}\mbox{ is a regular }(n+1)\mbox{-simplex}\bigr\}. \]
Note the closure of $W$ in $\vrleq{S^n}{r_n}$ is homeomorphic to $D^{n+1} \times \tfrac{\so(n+1)}{A_{n+2}}$.

\begin{proposition}\label{prop:S^n-small}
The maps $\pi f\colon\vrm{S^n}{r}\to S^n$ for $0<r<r_n$, $\pi f\colon\vrmless{S^n}{r_n}\to S^n$, and $\pi f\colon\vrmleq{S^n}{r_n}\setminus W\to S^n$ exist and are homotopy equivalences.
\end{proposition}

\begin{proof}
An identical proof works for all three maps, and hence we let $Y$ denote either $\vrm{S^n}{r}$ for $0<r<r_n$, or $\vrmless{S^n}{r_n}$, or $\vrmleq{S^n}{r_n}\setminus W$. The composition $\pi f \colon Y \to S^n$ is defined because $f(y)\neq\vec{0}$ for $\mu\in Y$ by the proof of~\cite[Lemma~3]{lovasz1983self}. Map $\pi f$ is continuous since $\pi$ and $f$ are.

We will show that the homotopy inverse to $\pi f$ is the inclusion map $i \colon S^n \hookrightarrow Y$, which is continuous only since $Y$ is equipped with the Wasserstein metric. Clearly $\pi f \circ i=id_{S^n}$. In order to show $i \circ \pi f \simeq id_{Y}$, define homotopy $H \colon Y \times I \to Y$ by $H(\mu,t)=(1-t)\mu+t\cdot i\pi f(\mu)$. This map is well-defined because if $\mu=\sum_{i=0}^k \lambda_i x_i$, then $\{x_0, \ldots, x_k, \pi f(\mu)\}$ is a simplex in $Y$;~see~\cite{AM}. Since $\pi f$ is continuous, map $H$ is continuous by Lemma~\ref{lem:homotopy-cont}. Since $H(-,0)=id_{Y}$ and $H(-,1)=i \circ \pi f$, we have shown $i \circ \pi f \simeq id_{Y}$. Therefore $\pi f\colon Y\to S^n$ is a homotopy equivalence.
\end{proof}

\begin{theorem}\label{thm:S^n-critical}
We have a homotopy equivalence $\vrmleq{S^n}{r_n}\simeq \Sigma^{n+1}\ \tfrac{\so(n+1)}{A_{n+2}}$.
\end{theorem}

\begin{proof}
We will construct the following commutative diagram.
\begin{center}
\begin{tikzpicture}[description/.style={fill=white,inner sep=2pt}] 
\matrix (m) [matrix of math nodes, row sep=3em, 
column sep=4em, text height=1.5ex, text depth=0.25ex] 
{  
D^{n+1} \times \frac{\so(n+1)}{A_{n+2}} \supset S^n \times \frac{\so(n+1)}{A_{n+2}} & S^n\\
D^{n+1} \times \frac{\so(n+1)}{A_{n+2}} \supset S^n \times \frac{\so(n+1)}{A_{n+2}} & \vrmleq{S^n}{r_n}\setminus W\\
};
\path[->,font=\scriptsize]
(m-1-1) edge node[above] {$h$} (m-1-2)
(m-1-1) edge node[below] {$h(x,y)=x$} (m-1-2)
(m-2-1) edge node[auto] {$g$} (m-2-2)
(m-2-1) edge node[auto] {$=$} (m-1-1)
(m-2-2) edge node[auto] {$\pi f$} (m-1-2)
;
\path[,font=\scriptsize]
(m-2-2) edge node[right] {$\simeq$} (m-1-2)
;
\end{tikzpicture}
\end{center}
For $y \in \tfrac{\so(n+1)}{A_{n+2}}$ let $\{y_0,\ldots,y_{n+1}\}$ be the $n+2$ vertices of the rotated regular $(n+1)$-simplex parameterized by $y$, and let
\[ \partial\Delta_y=\Biggl\{\sum_{i=0}^{n+1} \lambda_i y_i\in\vrmleq{S^n}{r_n}\setminus W~\Big|~\lambda_i=0\mbox{ for some }i\Biggr\} \] 
be the boundary of the corresponding simplex. Note $\pi f|_{\partial\Delta_y} \colon \partial\Delta_y \to S^n$ is bijective. Define map $g\colon S^n \times \frac{\so(n+1)}{A_{n+2}}\to\vrmleq{S^n}{r_n}\setminus W$  by letting $g(x,y)$ be the unique point of $\partial\Delta_y$ such that $\pi f(g(x,y))=x$; that is, $g(x,y)=(\pi f|_{\partial\Delta_y})^{-1}(x)$. Note $\pi f\circ g=h$, meaning the square commutes.

We now have the following sequence of homotopy equivalences.
\begin{align*}
\vrmleq{S^n}{r_n} &= \vrmleq{S^n}{r_n}\setminus W \cup_g \bigl(D^{n+1} \times \tfrac{\so(n+1)}{A_{n+2}}\bigr) \\
&\simeq S^n \cup_h \bigl(D^{n+1} \times \tfrac{\so(n+1)}{A_{n+2}}\bigr) \\
&\simeq \Bigl(S^n \times C\bigl(\tfrac{\so(n+1)}{A_{n+2}}\bigr)\Bigr) \cup_{S^n \times \tfrac{\so(n+1)}{A_{n+2}}} \Bigr(C(S^n) \times \tfrac{\so(n+1)}{A_{n+2}} \Bigl)\\
&= S^n * \tfrac{\so(n+1)}{A_{n+2}} = \Sigma^{n+1}\ \tfrac{\so(n+1)}{A_{n+2}}.
\end{align*}
Indeed, the first line is by the definition of $W$ and $g$. The second line follows from the commutative diagram above and the homotopy invariance properties of adjunction spaces (\cite[7.5.7]{brown2006topology} or~\cite[Proposition~5.3.3]{tom2008algebraic}). The third line follows from these same properties of adjunction spaces, induced by contractibility of $C(\tfrac{\so(n+1)}{A_{n+2}})$. The fourth line uses an equivalent definition for the join of two topological spaces as $Y*Z=Y\times C(Z)\cup_{Y\times Z}C(Y)\times Z$, and the fact that joining with a sphere gives an iterated suspension.
\end{proof}

\begin{remark}\label{rem:S1}
The case $n=1$ is instructive. We have $\tfrac{\so(2)}{A_3}=\tfrac{S^1}{\Z/3\Z}=S^1$, and hence the commutative diagram implies
\begin{align*}
\vrmleq{S^1}{r_1} &= \vrmleq{S^1}{r_1}\setminus W \cup_g (D^2 \times S^1) \simeq S^1 \cup_h (D^2 \times S^1) \\
&\simeq (S^1 \times D^2) \cup_{S^1\times S^1} (D^2 \times S^1) = S^1 * S^1 = S^3.
\end{align*}
\end{remark}

\begin{remark}
Since $\so(n+1)$ is a compact Lie group of dimension $\binom{n+1}{2}$, we have
\[ \tilde{H}_i\bigl(\vrmleq{S^n}{r_n}\bigr)=\tilde{H}_i\Bigl(\Sigma^{n+1}\ \tfrac{\so(n+1)}{A_{n+2}}\Bigr)=\tilde{H}_{i-n-1}\Bigl(\tfrac{\so(n+1)}{A_{n+2}}\Bigr)=
\begin{cases}
0 & i\le n+1\\
\Z & i=\binom{n+1}{2}+n+1.
\end{cases} \]

For the case $n=2$, let $T=A_4$ be the tetrahedral group. Let $2T$ be the binary tetrahedral group; it has $24$ elements but is isomorphic to neither $S_4$ nor $T \times \Z/2\Z$. The spherical 3-manifold $\frac{\so(3)}{T} = \frac{S^3}{2T}$ has fundamental group isomorphic to the binary tetrahedral group $2T$ with abelianization~$\Z/3$. Since $\frac{\so(3)}{T}$ is a 3-dimensional closed orientable manifold, this gives
\[ \tilde{H}_i\bigl(\vrmleq{S^2}{r_2}\bigr)=\tilde{H}_{i-3}\Bigl(\tfrac{\so(3)}{T}\Bigr) = \begin{cases} 
\Z/3 & i=4\\
\Z & i=6\\
0 & \mbox{otherwise.}
\end{cases} \]
\end{remark}

\begin{conjecture}
We conjecture that for all $n$, there exists an $\varepsilon>0$ such that
\begin{itemize}
\item $\vrmleq{S^n}{r}$ is homotopy equivalent to $\Sigma^{n+1}\ \tfrac{\so(n+1)}{A_{n+2}}$ for all $r_n\le r<r+\varepsilon$, and 
\item $\vrmless{S^n}{r}$, $\vrless{S^n}{r}$, and $\vrleq{S^n}{r}$ are homotopy equivalent to $\Sigma^{n+1}\ \tfrac{\so(n+1)}{A_{n+2}}$ for all $r_n<r<r+\varepsilon$.
\end{itemize}
\end{conjecture}

\section{Maps between Vietoris--Rips complexes and thickenings}\label{sec:maps}

Let $X$ be a metric space and $K$ a simplicial complex on vertex set $X$. We study some of the basic relationships between the simplicial complex $K$ and the metric thickening $K^m$. Note there is a natural bijection from the geometric realization of $K$ to $K^m$, given by $\sum \lambda_ix_i\mapsto\sum \lambda_i\delta_{x_i}$. We will show that the function $K\to K^m$ is continuous, a homeomorphism if $X$ is finite, and a weak equivalence if $X$ is discrete (Property~\ref{prop:continuous}--\ref{prop:weak-equiv}). The reverse function $K^m\to K$ is continuous if and only if $K$ is locally finite (Property~\ref{prop:locally-finite}). By an abuse of notation, we will let $\mu$ denote both $\sum \lambda_i\delta_{x_i}$ and $\sum \lambda_ix_i$.

As corollaries, we deduce that in the restricted setting of finite metric spaces, there are metric analogues of Latschev's theorem and the stability of persistent homology (Corollaries~\ref{cor:metric-Latschev} and \ref{cor:stability}). We conjecture that these theorems hold also in the case of arbitrary metric spaces (Conjecture~\ref{conj:metric-Latschev} and Conjecture~\ref{conj:stability}), but this is currently unknown. 

One could also compare the spaces considered in this section to the \emph{metric of barycentric coordinates} in~\cite[Section~7A.5]{bridson2011metric} and~\cite{dowker1952topology}. A related construction is given in~\cite{marin}.

\begin{proposition}\label{prop:continuous}
If $K^m$ is a metric thickening, then the function $K\to K^m$ is continuous.
\end{proposition}

\begin{proof}

The simplicial complex $K$ is equipped with the coherent topology, and therefore a map $K\to K^m$ is continuous if and only if its restriction $\sigma \to K^m$ is continuous for each closed simplex $\sigma$ in $K$. If $\sigma$ is a $k$-simplex, then its topology is induced from the Euclidean metric after embedding $\sigma$ into $\R^{k+1}$ in the standard way. After giving $\sigma$ this metric, one can show that $\sigma \to K^m$ is Lipschitz and hence continuous.
\end{proof}

\begin{proposition}\label{prop:homeo}
If $K^m$ is a metric thickening of a finite metric space $X$, then the function $K\to K^m$ is a homeomorphism.
\end{proposition}

\begin{proof}
Denote the function by $h\colon K\to K^m$. For any closed simplex $\sigma$ in $K$, note $h(\sigma)$ is closed in $K^m$ since a convergent sequence of points in $h(\sigma)$ must converge to a point in $h(\sigma)$. In the proof of Property~\ref{prop:continuous} we showed that each map $h\colon \sigma\to h(\sigma)$ is Lipschitz after equipping $\sigma$ with the Euclidean metric, but more is true: the map $h\colon \sigma\to h(\sigma)$ is bi-Lipschitz and hence a homeomorphism. Therefore $h^{-1}\colon K^m\to K$ is continuous as it is formed by gluing together continuous maps on a finite number of closed sets $h(\sigma)$.
\end{proof}

\begin{proposition}\label{prop:locally-finite}
Let $K^m$ be a metric $r$-thickening. The map $K^m\to K$ is continuous if and only if $K$ is locally finite.
\end{proposition}

\begin{proof}
Suppose $K$ is locally finite. Let $\{\mu_j\}$ be a sequence converging to $\mu=\sum_{i=0}^k \lambda_i x_i$ in $K^m$. Define 
\[ Y=\{x\in X~|~\mbox{there exists }i\mbox{ and }x'\in X\mbox{ with }d(x_i,x'),d(x',x)\le r\}. \]
Note that $Y$ is finite since local finiteness of $K$ implies there are only a finite number points $x'$ within distance $r$ from any $x_i$, and also only a finite number of points $x$ within distance $r$ from any $x'$. There exists some $J$ such that $j\ge J$ implies $\mu_j\in K[Y]^m$, since for any $\mu'\notin K[Y]^m$ we have $d_{K^m}(\mu',\mu)>r$. Since $\{\mu_j\}_{j\ge J}$ converges to $\mu$ in $K[Y]^m$, Property~\ref{prop:homeo} implies that $\{i(\mu_j)\}_{j\ge J}$ converges to $i(\mu)$ in $|K[Y]|$, and hence also in $K$.

For the reverse direction, suppose $K$ is not locally finite. Then some point $x\in X$ is contained in an infinite number of simplices, and hence in an infinite number of edges. Let $x_1, x_2, x_3, \ldots$ be a sequence of distinct vertices in $X$ such that each $[x,x_i]$ is an edge of $K$. The sequence $(1-\frac{1}{i})\delta_x+\frac{1}{i}\delta_{u_i}$ converges to $\delta_x$ in $K^m$. However, this sequence does not converge to $x$ in the geometric realization of $K$. Indeed, any convergent sequence together with its limit point is compact, but the set $\{x\}\cup\{(1-\frac{1}{i})x+\frac{1}{i}{x_i}~|~i\ge 1\}$ is not compact in the geometric realization of $K$ because it is not contained in a finite union of simplices.
\end{proof}

Combining Property~\ref{prop:continuous} and Property~\ref{prop:locally-finite} gives the following.

\begin{corollary}\label{cor:homeo}
If $K^m$ is a metric $r$-thickening and $K$ is locally finite, then the function $K\to K^m$ is a homeomorphism.
\end{corollary}

\begin{remark}\label{rem:nerve}
We now develop machinery which will allow us to understand the homotopy type of $K^m$ using the nerve lemma. The argument is similar to that used in~\cite[Theorem~5.2]{AFV} to prove a version of Hausmann's theorem for simplicial complexes embedded in Euclidean space.

Let~$X$ be an arbitrary metric space and let $K$ be a simplicial complex on vertex set $X$. For $\sigma\subseteq X$, denote by 
\[ \st(\sigma) = \Bigl\{\sum_{x \in \tau} \lambda_x x \in K^m~|~\sigma\subseteq\tau,\ \tau \in K,\ \lambda_x>0,\ \sum\lambda_i=1 \Bigr\} \]
the \emph{star} of~$\sigma$. We will consider the covering of $K^m$ by the sets $\nU=\{\st(x)~|~x \in X\}$. For any finite $\sigma \subseteq X$ we have that $\bigcap_{x \in \sigma} \st(x) = \st(\sigma)$. Note $\st(\sigma) = \emptyset$ if and only if $\sigma \notin K$; this implies that the nerve of $\nU$ is isomorphic to~$K$. If $\st(\sigma)$ is nonempty it is contractible: let $\mu=\sum_{x \in \sigma} \frac{1}{|\sigma|}x$ be the barycenter of~$\sigma$, then $\st(\sigma) \times [0,1] \longrightarrow \st(\sigma)$ defined by 
\[ \Bigl(\sum_{x \in \tau} \lambda_x x, t\Bigr) \mapsto t\mu+(1-t)\sum_{x \in \tau} \lambda_x x \] 
is a deformation retraction to the point~$\mu$. It follows that the sets in $\nU$ and their finite intersections are either empty or contractible. Hence whenever the nerve lemma applies to the covering~$\nU$, we get as a consequence that $K$ is homotopy equivalent to $K^m$. In particular, since every metric space is paracompact~\cite{stone1948paracompactness,ornstein1969new}, the nerve lemma as stated in Corollary~4G.3 of~\cite{Hatcher} applies whenever $\nU$ is a family of open sets. 
\end{remark}

If $X$ is discrete then the sets $\st(x)$ are open, and we obtain:

\begin{proposition}\label{prop:weak-equiv}
If $K^m$ is a metric thickening of a discrete metric space $X$, then $K\simeq K^m$.
\end{proposition}

\begin{proof}
Fix $x\in X$; we only need to show that $\st(x)$ is open. Since $X$ is discrete, there exists some $\varepsilon(x)>0$ such that $B_X(x,\varepsilon)=\{x\}$. Given $\mu\in\st(x)$, let $\lambda_x>0$ be the coefficient of $x$. Note that any point in the open ball $B_{K^m}(\mu,\lambda_x \varepsilon)$ must contain a positive coefficient for $x$. It follows that $ B_{K^m}(\mu,\lambda_x \varepsilon) \subseteq \st(x)$, and hence $\st(x)$ is open in $K^m$.
\end{proof}

\begin{remark}
For $X=\R$ (which is not discrete), the open star of $0$ is not open in $\vrmleq{\R}{r}$. Indeed, note that the sequence $\frac{1}{2}\delta_{1/i}+\frac{1}{2}\delta_{r+1/i}$ converges to $\frac{1}{2}\delta_0+\frac{1}{2}\delta_r$ in $\vrmleq{\R}{r}$. The limit point of this sequence is in the open star of $0$ even though no point of the sequence is.
\end{remark}

We end this section with three corollaries of Property~\ref{prop:
homeo}. Latschev's Theorem~\cite{Latschev2001} states that if $M$ is a closed Riemannian manifold, then for any $r$ sufficiently small there exists a $\delta$ (depending on $r$) such that if $d_\gh(M,X)<\delta$, then $\vrless{X}{r}\simeq M$. We immediately get an analogue of Latschev's theorem for Vietoris--Rips thickenings when $X$ is finite.

\begin{corollary}[Metric Latschev's theorem]\label{cor:metric-Latschev}
If $M$ is a closed Riemannian manifold, then for any $r$ sufficiently small there exists a $\delta$ such that if \emph{finite} $X$ satisfies $d_\gh(M,X)<\delta$, then $\vrmless{X}{r}\simeq M$.
\end{corollary}

\begin{proof}
Since $X$ is finite, Property~\ref{prop:homeo} gives that $\vrmless{X}{r}$ and $\vrless{X}{r}$ are homeomorphic. Then we have $\vrless{X}{r}\simeq M$ by Latschev's theorem~\cite{Latschev2001}.
\end{proof}

\begin{conjecture}\label{conj:metric-Latschev}
We conjecture that Corollary~\ref{cor:metric-Latschev} is true even for $X$ infinite.
\end{conjecture}

The next two corollaries relate to persistent homology~\cite{EdelsbrunnerHarer,chazal2009proximity}. If $X$ a metric space, then $\vrm{X}{-}$, $\vr{X}{-}$, $\cechm{X}{-}$, and $\cech{X}{-}$ define filtered topological spaces (since $r\le r'$ implies $\vrm{X}{r}\subseteq\vrm{X}{r'}$). We let $\ph_i(\vrm{X}{-})$ denote the persistent homology module of $\vrm{X}{-}$ in homological dimension $i\ge 0$.

\begin{corollary}\label{cor:PH}
Let $X$ be a finite metric space. For any dimension $i\ge 0$ the persistent homology modules $\ph_i(\vrm{X}{-})$ and $\ph_i(\vr{X}{-})$ are isomorphic, and the persistent homology modules $\ph_i(\cechm{X}{-})$ and $\ph_i(\cech{X}{-})$ are isomorphic.
\end{corollary}

\begin{proof}
Note that if $r\le r'$, then the following diagrams commute.
\begin{center}
\begin{tikzpicture}[description/.style={fill=white,inner sep=2pt}] 
\matrix (m) [matrix of math nodes, row sep=3em, 
column sep=3em, text height=1.5ex, text depth=0.25ex] 
{  
\vr{X}{r} & \vr{X}{r'}\\
\vrm{X}{r} & \vrm{X}{r'}\\
};
\path[->,font=\scriptsize]
(m-1-1) edge node[auto] {$\cong$} (m-2-1)
(m-1-2) edge node[auto] {$\cong$} (m-2-2)
;
\path[right hook->,font=\scriptsize]
(m-1-1) edge node[auto] {$$} (m-1-2)
(m-2-1) edge node[auto] {$$} (m-2-2)
;
\end{tikzpicture}
\quad\quad
\begin{tikzpicture}[description/.style={fill=white,inner sep=2pt}] 
\matrix (m) [matrix of math nodes, row sep=3em, 
column sep=4em, text height=1.5ex, text depth=0.25ex] 
{  
\cech{X}{r} & \cech{X}{r'}\\
\cechm{X}{r} & \cechm{X}{r'}\\
};
\path[->,font=\scriptsize]
(m-1-1) edge node[auto] {$\cong$} (m-2-1)
(m-1-2) edge node[auto] {$\cong$} (m-2-2)
;
\path[right hook->,font=\scriptsize]
(m-1-1) edge node[auto] {$$} (m-1-2)
(m-2-1) edge node[auto] {$$} (m-2-2)
;
\end{tikzpicture}
\end{center}
The result now follows from Property~\ref{prop:homeo}.
\end{proof}

\begin{remark}\label{rem:PH}
When the metric space $X$ is infinite, it is known that the persistent homology modules $\ph_i(\vrmleq{X}{-})$ and $\ph_i(\vrleq{X}{-})$ can differ at the endpoints of intervals. For example, $\ph_3(\vrmleq{S^1}{-})$ consists of the single half-open interval $[\frac{1}{3},\frac{2}{5})$, whereas $\ph_3(\vrleq{S^1}{-})$ consists of the single open interval $(\frac{1}{3},\frac{2}{5})$.
\end{remark}

\begin{conjecture}\label{conj:PH}
We conjecture that for $M$ a Riemannian manifold, the persistent homology modules $\ph_i(\vrm{M}{-})$ and $\ph_i(\vr{M}{-})$ are identical up to replacing closed interval endpoints with open endpoints, or vice-versa.
\end{conjecture}

\begin{corollary}\label{cor:stability}
Let $X$ and $Y$ be finite metric spaces and $i\ge 0$. Then 
\begin{align*}
d_b(\ph_i(\vrm{X}{-},\ph_i(\vrm{Y}{-})) &\le 2d_\gh(X,Y) \\
d_b(\ph_i(\cechm{X}{-},\ph_i(\cechm{Y}{-})) &\le 2d_\gh(X,Y),
\end{align*}
where $d_b$ denotes the bottleneck distance between persistent homology modules.
\end{corollary}

\begin{proof}
This follows from Corollary~\ref{cor:PH} and~\cite[Theorem~5.2]{ChazalDeSilvaOudot2013}.
\end{proof}

\begin{conjecture}\label{conj:stability}
We conjecture that Corollary~\ref{cor:stability} is true even for $X$ and $Y$ infinite.
\end{conjecture}

\section{Acknowledgements}

We would like to thank the anonymous referees, whose comments helped us clarify the exposition, expand some results, and simplify the proofs.

\appendix

\section{Metric thickenings and the Gromov--Hausdorff distance}

In this section we study how the Vietoris--Rips thickening behaves with respect to the Gromov--Hausdorff distance between metric spaces. The first result along these lines is Lemma~\ref{lem:thickening}, which implies that the Gromov-Hausdorff distance between $X$ and $\vrm{X}{r}$ is at most $r$.

If $X$ and $Y$ are isometrically embedded in a common metric space $Z$, then we let $d_\h^Z(X,Y)$ denote the Hausdorff distance between $X$ and $Y$ in $Z$.

\begin{lemma}
If $X$ and $Y$ are metric spaces with $d_\gh(X,Y)=\varepsilon$, then
\begin{align*}
d_\gh(\vrm{X}{r},\vrm{Y}{r}) &\le r+\varepsilon \\
d_\gh(\cechm{X}{r},\cechm{Y}{r}) &\le r+\varepsilon
\end{align*}
\end{lemma}

\begin{proof}
Let $\delta>0$. 
Let $Z$ be a metric space containing isometric embeddings of $X$ and $Y$ with $d_\h^Z(X,Y)\le\varepsilon+\delta$. Note that $\vrm{Z}{r}$ contains isometric copies of $\vrm{X}{r}$ and $\vrm{Y}{r}$. Given $\mu\in\vrm{X}{r}$, note by Lemma~\ref{lem:thickening} we have
\begin{align*}
d_{\vrm{Z}{r}}(\mu,\vrm{Y}{r})&\le d_{\vrm{Z}{r}}(\mu,Y)\le d_{\vrm{Z}{r}}(\mu,X)+d_{\vrm{Z}{r}}(X,Y) \\
&=d_{\vrm{X}{r}}(\mu,X)+d_\h^Z(X,Y)\le r+\varepsilon+\delta.
\end{align*}
The same argument works symmetrically for any $\mu\in\vrm{Y}{r}$, giving
\[ d_\h^{\vrm{Z}{r}}(\vrm{X}{r},\vrm{Y}{r})\le r+\varepsilon+\delta. \]
Since such a $Z$ exists for every $\delta>0$, this shows that $d_\gh(\vrm{X}{r},\vrm{Y}{r})\le r+\varepsilon$. An identical argument works for \v Cech complexes.
\end{proof}

\begin{remark}
The dependence of this bound on $r$ cannot be completely removed. Consider $X=\{\frac{-r-\varepsilon}{2},\frac{r+\varepsilon}{2}\}$ and $Y=\{\frac{-r+\varepsilon}{2},\frac{r-\varepsilon}{2}\}$. We have $d_\gh(X,Y)=\varepsilon$. Note we have $\vrm{X}{r}=X$ and $\vrm{Y}{r}=[\frac{-r+\varepsilon}{2},\frac{r-\varepsilon}{2}]$, and hence we have $d_\gh(\vrm{X}{r},\vrm{Y}{r})=\frac{r+\varepsilon}{2}$.
\end{remark}

\begin{lemma}
If $X$ and $Y$ are metric spaces, then
\[d_\gh(\vrm{X}{\infty},\vrm{Y}{\infty})=d_\gh(X,Y).\]
\end{lemma}

\begin{proof}
Let $Z$ be a metric space equipped with isometric embeddings of $X$ and $Y$ such that $d_\h^Z(X,Y)\le d_\gh(X,Y)+\varepsilon$. Note that $\vrm{Z}{\infty}$ contains isometric copies of $\vrm{X}{\infty}$ and $\vrm{Y}{\infty}$. If $\sum \lambda_i x_i \in \vrm{X}{\infty}$, then there exist points $y_i \in Y$ with $d_Z(x_i,y_i)\le d_\gh(X,Y)+\varepsilon$, and therefore the distance in $\vrm{Z}{\infty}$ between $\sum \lambda_i x_i$ and $\sum \lambda_i y_i\in\vrm{Y}{\infty}$ is at most $d_\gh(X,Y)+\varepsilon$. It follows that
\[ d_\h^{\vrm{Z}{\infty}}(\vrm{X}{\infty},\vrm{Y}{\infty})\le d_\gh(X,Y)+\varepsilon. \]
Since such a $Z$ exists for every $\varepsilon>0$, this shows $d_\gh(\vrm{X}{\infty},\vrm{Y}{\infty})\le d_\gh(X,Y)$.

For the converse direction, let $Z$ be any metric space equipped with isometric embeddings of $\vrm{X}{\infty}$ and $\vrm{Y}{\infty}$. Note that $Z$ also contains isometric embeddings of $X$ and $Y$. Without loss of generality, there is a point $x\in X$ with $d_Z(x,y)\ge d_\h^Z(X,Y)$ for all $y\in Y$ (perhaps after interchanging $X$ with $Y$). It follows that for the corresponding Dirac delta mass $\delta_x\in\vrm{X}{\infty}$, we have $d_Z(\delta_x,\sum \lambda_i y_i)\ge d_\h^Z(X,Y)$ for all $\sum \lambda_i y_i\in\vrm{Y}{\infty}$. Hence $d_\h^Z(\vrm{X}{\infty},\vrm{Y}{\infty})\ge d_\h^Z(X,Y)\ge d_\gh(X,Y)$. Since this is true for all metric spaces $Z$ equipped with isometric embeddings of $\vrm{X}{\infty}$ and $\vrm{Y}{\infty}$, we obtain $d_\gh(\vrm{X}{\infty},\vrm{Y}{\infty})\ge d_\gh(X,Y)$.
\end{proof}

\section{Crushings}\label{sec:crushings}

Let $X$ be a metric space and let $A\subseteq X$ be equipped with the subspace metric. The goal of this section is to prove that if there is a crushing of bounded speed from $X$ onto $A$, then the inclusion maps $\vrm{A}{r}\hookrightarrow\vrm{X}{r}$ and $\cechm{A}{r}\hookrightarrow\cechm{X}{r}$ are homotopy equivalences. 

Recall from~\cite{Hausmann1995} that a \emph{crushing} from $X$ onto $A$ is a continuous map $F \colon X \times [0,1] \to X$ satisfying
\begin{itemize}
\item[(i)] $F(x,1)=x$, $F(x,0)\in A$, $F(a,t)=a$ if $a\in A$, and
\item[(ii)] $d(F(x,t'),F(y,t'))\le d(F(x,t),F(y,t))$ whenever $t'\le t$.
\end{itemize}
For notational convenience, define $f_t\colon X\to X$ by $f_t(x)=F(x,t)$. We say that a crushing has \emph{bounded speed $c$} if $d(f_t(x),f_{t'}(x))\le c|t-t'|$ for all $x\in X$.

\begin{lemma}\label{lem:crushing}
If there is a crushing of bounded speed from $X$ onto a subset $A \subseteq X$, then the inclusion maps $\vrm{A}{r}\hookrightarrow\vrm{X}{r}$ and $\cechm{A}{r}\hookrightarrow\cechm{X}{r}$ are homotopy equivalences.
\end{lemma}

\begin{proof}
Let $F$ be a crushing from $X$ to $A$ of bounded speed $c$. Note that (ii) implies $d(f_t(x),f_t(y))\le d(x,y)$, and hence each $f_t$ is 1-Lipschitz.

Let $\iota\colon A\hookrightarrow X$ denote the inclusion map. Since $\iota$ and $f_0\colon X \to A$ are 1-Lipschitz, the maps $\tilde{\iota}\colon\vrm{A}{r}\hookrightarrow\vrm{X}{r}$ and $\tilde{f_0}\colon\vrm{X}{r}\to\vrm{A}{r}$ are defined and continuous by Lemma~\ref{lem:map-lipschitz}. We will show that $\tilde{\iota}$ and $\tilde{f_0}$ are homotopy inverses. Since $f_0\circ\iota=\id_A$, it follows that $\tilde{f_0}\circ\tilde{\iota}=\id_{\vrm{A}{r}}$.

We must show $\tilde{\iota}\circ\tilde{f_0}\simeq\id_{\vrm{X}{r}}$. Consider the map $\tilde{F} \colon \vrm{X}{r} \times [0,1] \to \vrm{X}{r}$ defined by $\tilde{F}(\cdot,t)=\tilde{f_t}$, which is well-defined since each $f_t$ is 1-Lipschitz. Let $\varepsilon>0$, and suppose $d(\sum \lambda_i x_i,\sum \lambda'_j x'_j)\le\varepsilon$ and $|t-t'|\le\varepsilon$. Then there is some $\pi_{i,j}\ge0$ with $\sum_j \pi_{i,j}=\lambda_i$, $\sum_i \pi_{i,j}=\lambda'_j$, and $\sum \pi_{i,j}d(x_i,x'_j)\le\varepsilon$. We have
\begin{align*}
&d\bigl(\tilde{f_t}(\sum \lambda_i x_i),\tilde{f_{t'}}(\sum \lambda'_j x'_j)\bigr)\\
=\ &d\bigl(\sum \lambda_if_t(x_i),\sum \lambda'_jf_{t'}(x'_j)\bigr)\\
\le\ &\sum_{i,j} \pi_{i,j}d\bigl(f_t(x_i),f_{t'}(x'_j)\bigr)\\
\le\ &\sum_{i,j} \pi_{i,j}d\bigl(f_t(x_i),f_t(x'_j)\bigr)+\sum_{i,j} \pi_{i,j}d\bigl(f_t(x'_j),f_{t'}(x'_j)\bigr)\quad\mbox{by the triangle inequality}\\
\le\ &\sum_{i,j}\pi_{i,j} d(x_i,x'_j)+c\sum_{i,j}\pi_{i,j}|t-t'| \quad\mbox{since }f_t\mbox{ is 1-Lipschitz and }F\mbox{ has bounded speed }c\\
\le\ &\varepsilon+c|t-t'|\le (1+c)\varepsilon,
\end{align*}
and so $\tilde{F}$ is continuous. Since $\tilde{F}(\cdot,0)=\tilde{\iota}\circ\tilde{f_0}$ and $\tilde{F}(\cdot,1)=\id_{\vrm{X}{r}}$, it follows that $\tilde{\iota}\circ\tilde{f_0}\simeq\id_{\vrm{X}{r}}$.

The proof for the \v Cech case is identical.
\end{proof}

\section{Examples where $H_*(\vrleq{X}{r})$ has infinite rank but $H_*(\vrmleq{X}{r})$ does not}\label{app:Droz}

\begin{example}
Let $X=[0,1]\times\{0,1\}\subseteq\R^2$. In~\cite[Section~5.2.1]{ChazalDeSilvaOudot2013} it is remarked that $\vrleq{X}{1}$ with the standard metric has uncountable 1-dimensional homology.

We now show that $\vrmleq{X}{1}$ with the Wasserstein metric is contractible. Note that a crushing of bounded speed 1 from $X$ to $\{(0,0),(0,1)\}$ is given by $F\colon X\times [0,1]\to X$ defined via
\[F((s,0),t)=(ts,0)\quad\mbox{and}\quad F((s,1),t)=(ts,1).\]
Hence by Lemma~\ref{lem:crushing} we have
\[ \vrmleq{X}{r}\simeq\vrmleq{\{(0,0),(0,1)\}}{r}\simeq \begin{cases}
S^0 &\mbox{if }r<1\\
* &\mbox{if }r\ge1.
\end{cases}\]
\end{example}

\begin{example}
Let $X$ be two non-parallel rectangles, namely
\[X=\{(s,0,z)\in\R^3~|~s\in[0,2],z\in[0,1]\}\cup\{(s,1+\tfrac{1}{2}s,z)\in\R^3~|~s\in[0,2],z\in[0,1]\},\]
equipped with the $\ell^1$ metric. In~\cite[Proposition~5.9]{ChazalDeSilvaOudot2013} it is shown that $\vrleq{X}{r}$ with the standard metric has uncountable 1-dimensional homology for all $r\in[1,2]$.

We now show that $\vrmleq{X}{r}$ with the Wasserstein metric is either homotopy equivalent to $S^0$ or contractible. Note that a crushing of bounded speed 2 from $X$ onto $\{(0,0,0),(0,1,0)\}$ is given by $F\colon X\times [0,1]\to X$ defined via
\[F((s,0,z),t)=(ts,0,tz)\quad\mbox{and}\quad F((s,1+\tfrac{1}{2}s,z),t)=(ts,1+\tfrac{1}{2}ts,tz).\]
Hence by Lemma~\ref{lem:crushing} we have
\[ \vrmleq{X}{r}\simeq\vrmleq{\{(0,0,0),(0,1,0)\}}{r}\simeq \begin{cases}
S^0 &\mbox{if }r<1\\
* &\mbox{if }r\ge1.
\end{cases}\]
\end{example}

\begin{example}
Let $S^1$ be the circle of unit circumference equipped with the path-length metric; this gives $r_1=\frac{1}{3}$. In~\cite{AA-VRS1} it is shown that $\vrleq{S^1}{\frac{1}{3}}\simeq\bigvee^{\mathfrak{c}}S^2$, whereas $\vrmleq{S^1}{\frac{1}{3}}\simeq S^3$ by Theorem~\ref{thm:S^n-critical}.
\end{example}

\section{Vietoris--Rips thickenings are usually not complete}\label{sec:not-compact}

The following results show that the Vietoris--Rips thickening of a complete space need not be complete, and that the Vietoris--Rips thickening of a compact space need not be compact.

Let $X$ be a metric space and $K$ a simplicial complex on vertex set $X$. Recall the \emph{$n$-skeleton} of~$K^m$ is denoted $S_n(K^m) = \{\sum_{i=0}^k \lambda_i x_i \in K^m~|~ k\le n\}$. The sets $S_n(K^m)$ are a countable closed covering of~$K^m$. If the $S_n(K^m)$ have empty interior and hence are nowhere dense, then by Baire's category Theorem $K^m$ cannot be complete.

\begin{theorem}
If $X$ is a metric space without isolated points, then $\vrmless{X}{r}$ is not complete for all~${r > 0}$.
\end{theorem}

\begin{proof}
It suffices to show that the $S_n(K^m)$ have empty interior. Consider a point $\mu=\sum_{i=0}^k \lambda_i x_i \in S_n(K^m)$ with $\lambda_i>0$ and let $\varepsilon>0$ be arbitrary. Since $X$ has no isolated points, there exist disjoint points $y_0,\ldots,y_{n+1}\in X$ with $d(x_0,y_j)<\min\{\varepsilon,r-\diam(\{x_0,\ldots,x_k\})\}$ for all $0\le j\le n+1$. Note $\mu'=\sum_{j=0}^{n+1}\frac{\lambda_0}{n+2}y_j+\sum_{i=1}^k \lambda_i x_i$ satisfies $d(\mu',\mu)<\varepsilon$ and $\mu'\in\vrmless{X}{r}\setminus S_n(K^m)$.
\end{proof}

\begin{theorem}
If $X\subseteq \R^n$ is an infinite convex subset, then $\vrmleq{X}{r}$ is not complete for all~${r > 0}$.
\end{theorem}

\begin{proof}
We mimic the above proof. By assumption on $X$ there exist disjoint points $y_0,\ldots,y_{n+1}\in X$ with $d(x_0,y_j)<\varepsilon$ and $\diam(\{x_0,\ldots,x_k,y_0,\ldots,y_{n+1}\})\le \min\{\varepsilon,r\}$.
\end{proof}

\bibliographystyle{plain}
\bibliography{MetricReconstructionViaOptimalTransport}
\end{document}